\documentclass{elsarticle}

\usepackage{amsfonts}
\usepackage{amsmath,amscd,amssymb,amsthm}
\usepackage{mathrsfs}
\usepackage[cmtip,arrow]{xy}
\usepackage{pb-diagram,pb-xy}
\usepackage{enumerate}
\usepackage{verbatim}
\usepackage{color}

\newtheorem{theorem}{Theorem}
\theoremstyle{plain}
\newtheorem{lemma}[theorem]{Lemma}
\newtheorem{prop}[theorem]{Proposition}
\newtheorem{cor}[theorem]{Corollary}

\newtheorem{question}[theorem]{Question}
\theoremstyle{definition}
\newtheorem{definition}[theorem]{Definition}
\theoremstyle{remark}

\newtheorem{rmk}[theorem]{Remark}
\newtheorem{example}[theorem]{Example}

\newcommand{\R}{\mathbb{R}}
\newcommand{\Z}{\mathbb{Z}}
\newcommand{\F}{\mathbb{F}}

\begin{document}

\begin{frontmatter}

\title{Relations among the kernels and images of {S}teenrod squares
  acting on right $\mathcal{A}$-modules}

\author[]{Shaun V. Ault}\ead{ault@fordham.edu}

\address{Department of Mathematics \\ 
         Fordham University \\ 
         Bronx, New York, 10461, USA.}

\begin{abstract}  
  In this note, we examine the right action of the Steenrod algebra
  $\mathcal{A}$ on the homology groups $H_*(BV_s, \F_2)$, where $V_s =
  \F_2^s$.  We find a relationship between the intersection of kernels
  of $Sq^{2^i}$ and the intersection of images of $Sq^{2^{i+1}-1}$,
  which can be generalized to arbitrary right $\mathcal{A}$-modules.
  While it is easy to show that $\bigcap_{i=0}^{k}
  \mathrm{im}\,Sq^{2^{i+1}-1} \subseteq \bigcap_{i = 0}^k
  \mathrm{ker}\,Sq^{2^i}$ for any given $k \geq 0$, the reverse
  inclusion need not be true.  We develop the machinery of homotopy
  systems and null subspaces in order to address the natural question
  of when the reverse inclusion can be expected.  In the second half
  of the paper, we discuss some counter-examples to the reverse
  inclusion, for small values of $k$, that exist in $H_*(BV_s, \F_2)$.
\end{abstract}

\begin{keyword}
  Hit Problem \sep Steenrod Algebra \sep Homotopy System

  \MSC[2010] 55S10 \sep 55R40 \sep 57T25
\end{keyword}

\end{frontmatter}


\section{Introduction}\label{sec.intro}

For $s \geq 1$, let $V_s = \F_2^s$, the elementary Abelian $2$-group
of rank $s$.  The action of the Steenrod algebra $\mathcal{A}$ on the
cohomology of $BV_s$ has been a topic of much study
(see~\cite{Wood98,Wood00,Nam04}, for example).  The problem of finding
a minimal generating set for $H^*\left(BV_s, \F_2 \right)$, known as
the ``hit problem,'' seems currently out of reach, though there are
complete answers in the cases $s \leq 4$
(see~\cite{Board,Kameko,Kameko03,Sum07}).  We examine the dual
problem: finding a basis of the space of $\mathcal{A}^+$-annihilated
elements of the homology of $BV_s$.  An element $x$ is
``$\mathcal{A}^+$-annihilated,'' if $xSq^p = 0$ for every $p > 0$.
Some important work has already been done on this dual problem ({\it
  e.g.}~\cite{ACH,CH,RP}).  Let $\Gamma$ be the bigraded space
$\{\Gamma_{s,*}\}_{s \geq 0}$, where $\Gamma_{s,*} = H_*\left(BV_s,
\F_2 \right)$.  We write $\Gamma_{\mathcal{A}^+}$ for the space of
$\mathcal{A}^+$-annihilated elements of $\Gamma$.  The problem may be
phrased thus:

\medskip

{\bf $\mathcal{A}^+$-Annihilated Problem.}  {\it Find an
  $\mathbb{F}_2$ basis for $\Gamma_{\mathcal{A}^+}$.}

\medskip

We find the $\mathcal{A}^+$-annihilated problem to be more natural
than the hit problem, since $\Gamma_{\mathcal{A}^+}$ admits an algebra
structure; indeed $\Gamma_{\mathcal{A}^+}$ is free as associative
algebra~\cite{A1}.  William Singer and the author proposed to study
the spaces of ``partially $\mathcal{A}^+$-annihilated'' elements, as
these may be more accessible.  It was found that these spaces are also
free associative algebras~\cite{AS}.  Recall some of the definitions
and notations used in~\cite{AS}.  The bigraded algebra
$\widetilde{\Gamma} = \{\widetilde{\Gamma}_{s,*}\}_{s \geq 0}$ is
defined by:
\[
  \widetilde{\Gamma}_{s, *} = \left\{\begin{array}{ll} \F_2,
  \;\textrm{concentrated in internal degree $0$},\;& \textrm{if $s =
    0$},\\ 
  \widetilde{H}_*\left((\R P^{\infty})^{\wedge s},
  \F_2\right)), & \textrm{if $s \geq 1$}.
  \end{array}\right.
\]
Observe, the homotopy equivalence $BV_s \simeq \prod^s \R P^{\infty}$
connects this definition with that of $\Gamma$ given above.  The
Steenrod algebra acts on the right of $\widetilde{\Gamma}$, and in
this note, all functions that we define will be written on the right
of their arguments in order to maintain consistency.  

Let $\mathcal{S}_k$ be the Hopf subalgebra of $\mathcal{A}$ generated
by $\{Sq^{2^i} \}_{i \leq k}$.  For each $k \geq 0$, define the
bigraded space of partially $\mathcal{A}^+$-annihilated elements,
$\Delta(k) \subseteq \widetilde{\Gamma}$, by
\[
  \Delta(k) = \widetilde{\Gamma}_{\mathcal{S}^+_k} = \bigcap_{i = 0}^k
  \mathrm{ker}\,Sq^{2^i}.
\]
The definition generalizes to arbitrary $\mathcal{A}$-modules,
whether unstable or not.
\begin{definition}\label{def.Delta_k}
  Let $M$ be a right $\mathcal{A}$-module and $k \geq 0$.  Define the
  graded space of {\it partially $\mathcal{A}^+$-annihilated} elements
  of $M$,
  \[
    \Delta_M(k) = M_{\mathcal{S}^+_k} = \bigcap_{i = 0}^k
    \mathrm{ker}\,Sq^{2^i}.
  \]
\end{definition}
Define also the graded spaces of {\it simultaneous spike images},
\begin{definition}
  Let $M$ be a right $\mathcal{A}$-module and $k \geq 0$.  The
  space of {\it simultaneous spike images} of $M$ is defined
  by
  \[
    I_M(k) = \bigcap_{i=0}^{k} \mathrm{im}\,Sq^{2^{i+1}-1}.
  \]
  When $M = \widetilde{\Gamma}$, the notation $I(k)$ is used.
\end{definition}
During the course of investigating the ``$\mathcal{A}^+$-annihilated
problem'' in $\widetilde{\Gamma}$, Singer and the author made a number
of (unpublished) conjectures that we later proved false.  One of the
most tantalizing and longest-lived of our conjectures concerned a
proposed relationship between $\mathcal{S}^+_k$-annihilateds and
simultaneous spike images, based on the observation,
\begin{prop}\label{prop.wall}
  Let $M$ be a right $\mathcal{A}$-module.  Then $I_M(k) \subseteq
  \Delta_M(k)$.
\end{prop}
The proof is immediate from the Adem relations: $Sq^{2n-1}Sq^n = 0$,
for all $n \geq 1$.  It is natural to ask whether some version of the
converse of Prop.~\ref{prop.wall} could hold in $\widetilde{\Gamma}$.
Observe that if $x \in \widetilde{\Gamma}_{s,d}$ with $d < 2^{k+1}$,
then $x$ automatically lies in $\mathrm{ker}\,Sq^{2^k}$ due to
instability conditions in $\widetilde{\Gamma}_{s,*}$, so there is no
reason to suppose that such $x$ lie in the image of $Sq^{2^{k+1}-1}$.
However, in view of Prop.~\ref{prop.wall}, we may well ask:
\begin{question}\label{conj.imker}
  For each integer $k \geq 0$, if the internal degree of $x \in
  \widetilde{\Gamma}$ is at least $2^{k+1}$, then is it the case that
  $x \in \Delta(k)$ if and only if $x \in I(k)$?
\end{question}
The answer is {\it yes} when $k=0$, as shown in~\cite{AS}.  Moreover,
we had experimental evidence, including many hours of machine
computation, that supported a positive answer in general.  A
counterexample was first found for the case $k=2$, and subsequent work
was directed at proving the case $k=1$.  It wasn't until very recently
that the author produced a counterexample for $k=1$ in bidegree $(5,
9)$, and so the answer to Question~\ref{conj.imker} has to be {\it no}
for all $k \geq 1$.  In this note, we shall provide stronger
hypotheses that turn the question into a true theorem.  In order to
state and prove the result, we introduce (in
Section~\ref{sec.homotopy_systems}) the machinery of {\it homotopy
  systems} on right $\mathcal{A}$-modules.  In
Sections~\ref{sec.shift} and~\ref{sec.hom_sys_general}, we will define
a suitable homotopy system for $\widetilde{\Gamma}$ and use it and
some variations thereof to analyze a broad class class of
$\mathcal{A}^+$-annihilated elements in some important right
$\mathcal{A}$-modules (see Thm.~\ref{thm.im-ker_Gamma} for example).
Further observations about the structure of $\Delta_M(k)$ and $I_M(k)$
are considered in Section~\ref{sec.quotient}.
Section~\ref{sec.structure} contains some structure formulas for
elements of $\Delta(1)$, which lead to an explicit element $z \in
\Delta(1)_{5,9}$ that is not contained in $I(1)$ (we produce the
element in Section~\ref{sec.questions}).

\section{Homotopy Systems}\label{sec.homotopy_systems}

\subsection{Definition of Homotopy System.}

Let $M$ be a right $\mathcal{A}$-module.  Fix an integer $k \geq 0$. A
{\it $k^{th}$-order homotopy system} on $M$ consists of a subspace $N
\subseteq M$, called a {\it null subspace}, which may not necessary be
an $\mathcal{A}$-module, and for each integer $0 \leq m \leq k$, a
homomorphism of vector spaces $\psi^{2^m} : M \to M$ such that:
\begin{itemize}
  \item $N$ is stable under $\psi^{2^m}$.
  \item If $1 \leq m \leq k$ and $\ell < 2^m$, then there is
    a commutative diagram,
    \begin{equation}\label{eqn.comm_diag}
      \begin{diagram}
        \node{N}
        \arrow{e,t}{Sq^{\ell}}
        \arrow{s,l}{\psi^{2^m}}
        \node{M}
        \arrow{s,r}{\psi^{2^m}}
        \\
        \node{N}
        \arrow{e,t}{Sq^{\ell}}
        \node{M}
      \end{diagram}
    \end{equation}
  \item If $0 \leq m \leq k$, then there is a {\it homotopy
    relation},
    \begin{equation}\label{eqn.homotopy}
      \left( \psi^{2^m}Sq^{2^m} + Sq^{2^m}\psi^{2^m} \right) : N \to M
      \quad = \quad N \hookrightarrow M.
    \end{equation}
\end{itemize}
\begin{example}
  If $M$ is considered as a chain complex with differential $Sq^1$,
  and $N = M$, then a zeroth-order homotopy system on $M$ consists of
  a chain homotopy $\psi^1$ from the identity map to the zero map.
\end{example}
\begin{rmk}
  The null subspace $N$ is part of the definition, and there is no
  requirement that $N$ be the {\it maximal} subspace such that
  Eqns.~(\ref{eqn.comm_diag}),~(\ref{eqn.homotopy}) are valid with
  respect to it.  In fact, if $\{\psi^{2^m}\}_{m \leq k}$ is a
  $k^{th}$-order homotopy system on $M$ with null subspace $N$, then
  for any $j \leq k$, the set of maps $\{\psi^{2^m}\}_{m \leq j}$ is a
  $j^{th}$-order homotopy system on $M$ with the same null subspace
  $N$, although in practice we might be able to find a larger null
  subspace.
\end{rmk}

\subsection{The Main Theorem.}

\begin{theorem}\label{thm.imker}
  Let $M$ be a right $\mathcal{A}$-module and fix an integer $k \geq
  0$.  If there is a $k^{th}$-order homotopy system $\{\psi^{2^i}\}_{i
    \leq k}$ on $M$ with associated null subspace $N$, then
  \[
    N \cap \Delta_M(k) = N \cap I_M(k).
  \]
  Moreover, for $x \in N \cap \Delta_M(k)$ and for each $i \leq k$,
  the element
  \[
    y_i = x\psi^{2^{i}}\psi^{2^{i-1}}\cdots \psi^{2}\psi^1
  \]
  is a preimage of $x$ under $Sq^{2^{i+1}-1}$.
\end{theorem}
\begin{proof}
  We fix a right $\mathcal{A}$-module $M$ and the integer $k$.
  Proposition~\ref{prop.wall} implies $N \cap I_M(k) \subseteq N \cap
  \Delta_M(k)$.  Now suppose $x \in N \cap \Delta_M(k)$.  We shall
  prove that the elements $y_i = x\psi^{2^{i}}\psi^{2^{i-1}}\cdots
  \psi^{2}\psi^1$ are in fact preimages of $x$ under $Sq^{2^{i+1}-1}$
  by first showing that for each $0 \leq j \leq i+1$,
  \begin{equation}\label{eqn.induction_goal}
    y_iSq^{2^{i+1}-1} = x\psi^{2^i}\psi^{2^{i-1}}\cdots
    \psi^{2^{j+1}}\psi^{2^{j}}Sq^{2^j}Sq^{2^{j+1}} \cdots
    Sq^{2^{i-1}}Sq^{2^i},
  \end{equation}
  where we understand this to mean $y_iSq^{2^{i+1}-1} = x$ in the case
  $j = i+1$.  The verification of Eqn.~(\ref{eqn.induction_goal}) is
  by induction on $j$.  When $j = 0$, Eqn.~(\ref{eqn.induction_goal})
  simply follows by definition of $y_i$ and relations within
  $\mathcal{A}$:
  \begin{eqnarray*}
    y_iSq^{2^{i+1}-1} &=& x\psi^{2^{i}}\psi^{2^{i-1}}\cdots
    \psi^{2}\psi^1Sq^{2^{i+1}-1}\\
    &=& x\psi^{2^{i}}\psi^{2^{i-1}}\cdots
    \psi^{2}\psi^1Sq^1Sq^2 \cdots Sq^{2^{i-1}}Sq^{2^{i}}
  \end{eqnarray*}
  
  Next, suppose Eqn.~(\ref{eqn.induction_goal}) is verified for all
  numbers up to $j$ for some $j \leq i$.
  \begin{eqnarray*}
    y_iSq^{2^{i+1}-1} &=& x\psi^{2^i}\cdots
    \psi^{2^{j+1}}\psi^{2^{j}}Sq^{2^j}Sq^{2^{j+1}} \cdots Sq^{2^i}\\
    &=& \left(\left[x\psi^{2^{i}}\cdots
      \psi^{2^{j+1}}\right]\psi^{2^j}Sq^{2^j}\right)Sq^{2^{j+1}}\cdots Sq^{2^i}\\
    &=& \left(\left[x\psi^{2^{i}}\cdots
      \psi^{2^{j+1}}\right]Sq^{2^j}\psi^{2^j} + x\psi^{2^{i}}\cdots
    \psi^{2^{j+1}}\right)Sq^{2^{j+1}}\cdots Sq^{2^i}\\
    &=& \left(\left[xSq^{2^j}\psi^{2^{i}}\cdots
      \psi^{2^{j+1}}\right]\psi^{2^j} + x\psi^{2^{i}}\cdots
    \psi^{2^{j+1}}\right)Sq^{2^{j+1}}\cdots Sq^{2^i}\\
    &=& \left(0 + x\psi^{2^{i}}\cdots
    \psi^{2^{j+1}}\right)Sq^{2^{j+1}}\cdots Sq^{2^i}\\
    &=& x\psi^{2^{i}}\cdots \psi^{2^{j+1}}Sq^{2^{j+1}}\cdots
    Sq^{2^i}\\
  \end{eqnarray*}
  Note, we used the fact that $x \in N$, which is closed under each
  map $\psi^{2^i}$, $\psi^{2^{i-1}}$, $\ldots$, $\psi^{2^{j}}$, in
  order to apply relations~(\ref{eqn.comm_diag})
  and~(\ref{eqn.homotopy}).  We also used the fact that $x \in
  \Delta_M(k)$ to get $xSq^{2^j} = 0$.  The upshot of
  Eqn.~(\ref{eqn.induction_goal}) is that when $j = i+1$, it gives us
  exactly what we wanted:
  \[
    y_iSq^{2^{i+1}-1} = x.
  \]
  Therefore, for each $0 \leq i \leq k$, we have $x \in \mathrm{im}
  Sq^{2^{i+1}-1}$ with preimage $y_i$ as specified in the statement of
  the theorem.
\end{proof}

\begin{rmk}
  Though there may be many preimages for $x$, the elements $y_i$ have
  the added property of being members of the null subspace $N$.
\end{rmk}

\section{Shift Maps}\label{sec.shift}

In this section, we find certain $k^{th}$-order homotopy systems for
various important right $\mathcal{A}$-modules.  These homotopy systems
are based on {\it shift maps}, which will be defined presently for
elements of $\widetilde{\Gamma}$.  Recall, the generators of
$\widetilde{\Gamma}_{s,*}$ are $s$-length monomials in the
non-commuting symbols $\{\gamma_i\}_{i \geq 1}$, where each $\gamma_i
\in \widetilde{H}_i(\R P^{\infty})$ is the canonical generator.  For
convenience and readability in formulas, we generally denote a
monomial $\gamma_{a_1}\gamma_{a_2} \cdots \gamma_{a_s} \in
\widetilde{\Gamma}_{s,*}$ by $[a_1, a_2, \ldots, a_s]$.  For each $s
\geq 1$, $d \geq 0$, $r \geq 0$, and $1 \leq i \leq s$, define the
$i^{th}$-place shift maps, $\psi_i^{r} : \widetilde{\Gamma}_{s,d} \to
\widetilde{\Gamma}_{s,d+r}$ on generators by:
\[
  [a_1, a_2, \ldots, a_s]\psi_i^{r} = [a_1, a_2, \ldots, a_i + r,
    \ldots, a_s].
\]
We will find that (for a fixed $i$), the set of shifts
$\{\psi_i^{2^m}\}_{m \leq k}$ forms a $k^{th}$-order homotopy system
with a rather large null subspace.  We start by showing there are
certain commutation relations in $\widetilde{\Gamma}_{1,*}$.  The
binomial coefficient $\binom{a}{b}$ is taken modulo $2$, and unless
otherwise stated, $\binom{a}{b} = 0$ if either $a < 0$ or $b < 0$.  As
a consequence of Lucas' Theorem,
\begin{equation}\label{eqn.lucas}
  \binom{a}{2^n} = \binom{b}{2^n}, \quad
  \textrm{if $a, b \geq 0$ and $a \equiv b \;(\textrm{mod}\; 2^{n+1})$.}
\end{equation}

\begin{lemma}\label{lem.shift_commutation_1}
  For $m > n \geq 0$, if $a \geq 2^{n}$, 
  \[
    [a]\psi_1^{2^m}Sq^{2^n} = [a]Sq^{2^n}\psi_1^{2^m}
  \]
\end{lemma}
\begin{proof}
  Consider the expression on the left hand side:
  \begin{eqnarray}
    [a]\psi_1^{2^m}Sq^{2^n} &=& [a+2^m]Sq^{2^n}\\
    \label{eqn.second_binom}
     &=& \binom{a + 2^m - 2^n}{2^n}[a + 2^m - 2^n].
  \end{eqnarray}

  Since $m \geq n+1$, $a + 2^m - 2^n \equiv a - 2^n
  \;\mathrm{mod}\;2^{n+1}$.  By hypothesis, $a-2^n \geq 0$, and so by
  formula~(\ref{eqn.lucas}), there is equality $\binom{a+2^m-2^n}{2^n}
  = \binom{a-2^n}{2^n}$.

  \begin{eqnarray}
    [a]\psi_1^{2^m}Sq^{2^n} &=& \binom{a - 2^n}{2^n}[a+2^m - 2^n]\\
    &=& \binom{a-2^n}{2^n} [a - 2^n]\psi_1^{2^m}\\    
    &=& [a]Sq^{2^n}\psi_1^{2^m} 
  \end{eqnarray}
\end{proof}

\begin{cor}\label{cor.shift_comm_1}
  If $2^m > \ell \geq 0$ and $a \geq 2^{m-1}$,
  \[
    [a]\psi_1^{2^m}Sq^{\ell} = [a]Sq^{\ell}\psi_1^{2^m}.
  \]
\end{cor}
\begin{proof}
  Since $\ell < 2^m$, we have $Sq^{\ell} \in \mathcal{S}_{m - 1}$.  So
  $Sq^{\ell}$ can be written as a sum of products of ``2-power''
  squares $Sq^{2^{i}}$ such that $i < m$, each of which commutes with
  $\psi_1^{2^m}$ by Lemma~\ref{lem.shift_commutation_1}.
\end{proof}

We will have occasion to use a more precise formula for the binomial
coefficients:
\begin{equation}\label{eqn.binom}
  \textrm{For $a \geq 0$,} \quad \binom{a}{2^n} = 
  \left\{ \begin{array}{ll} 0, \;\mathrm{if} \; a
    \equiv 0, 1, \ldots, 2^n-1 \; (\mathrm{mod}\; 2^{n+1}),\\ 1,
    \;\mathrm{if} \; a \equiv 2^n, 2^n+1, \ldots, 2^{n+1}-1
    \;(\mathrm{mod}\; 2^{n+1}).\\
  \end{array}\right.
\end{equation}
  
\begin{lemma}\label{lem.shift_homotopy_1}
  For $m \geq 0$, if $a \geq 2^{m}$, 
  \[
    [a]\psi_1^{2^m}Sq^{2^m} + [a]Sq^{2^m}\psi_1^{2^m} = [a].
  \]
\end{lemma}
\begin{proof}
  The proof is a straightfoward computation:
  \begin{eqnarray}
    \lefteqn{[a]\psi_1^{2^m}Sq^{2^m} + [a]Sq^{2^m}\psi_1^{2^m}} 
    \nonumber\\
    &=& [a+2^m]Sq^{2^m} + \binom{a-2^m}{2^m} [a-2^m]\psi_1^{2^m}\\
    \label{eqn.sum_binom}
    &=& \binom{a}{2^m}[a]+ \binom{a-2^m}{2^m}[a].
  \end{eqnarray}
  By formula~(\ref{eqn.binom}), the expression
  in~(\ref{eqn.sum_binom}) is equal to $[a]$.
\end{proof}

These results will serve to prove that, under some conditions, the
same commutation relations hold in $\widetilde{\Gamma}_{s,*}$ for any
$s \geq 1$.
\begin{lemma}\label{lem.shift_commutation}
  If $2^m > \ell \geq 0$ and $a_i \geq 2^{m-1}$ in $x = [a_1, \ldots,
    a_i, \ldots, a_s]$, then
  \[
    x\psi_i^{2^m}Sq^{\ell} = xSq^{\ell}\psi_i^{2^m}.
  \]
\end{lemma}
\begin{proof}
  Since $Sq^{\ell}$ commutes with the action of the symmetric group
  $\Sigma_{s}$, and there is also a relation, $\psi_i^{r}\sigma =
  \sigma \psi_{\sigma(i)}^r$, for $\sigma \in \Sigma_{s}$, it is
  sufficient to prove the lemma for $i = 1$.
  Cor.~\ref{cor.shift_comm_1} proves the case $s = 1$, so assume $s >
  1$ and write $x = [a_1]\cdot z$, where $z = [a_2, \ldots, a_s]$.
  \begin{eqnarray}
    ([a_1] \cdot z)Sq^{\ell}\psi_1^{2^m}
    \label{eqn.cartan}
    &=& \sum_{p = 0}^{\ell}\left([a_1]Sq^p\cdot zSq^{\ell-p} \right)\psi_1^{2^m}\\ 
    \label{eqn.psi_1}
    &=& \sum_{p = 0}^{\ell}[a_1]Sq^p\psi_1^{2^m}\cdot zSq^{\ell-p}\\
    \label{eqn.psi_1_commute}
    &=& \sum_{p = 0}^{\ell} [a_1]\psi_1^{2^m}Sq^p\cdot zSq^{\ell-p}\\
    &=& \left( [a_1]\psi_1^{2^m} \cdot z \right)Sq^{\ell}\\
    \label{eqn.psi_1_again}
    &=& \left( [a_1] \cdot z \right)\psi_1^{2^m}Sq^{\ell}.
  \end{eqnarray}
  Lines (\ref{eqn.psi_1}) and (\ref{eqn.psi_1_again}) follow from the
  fact that $\psi_1^{2^m}$ acts only on the first factor of a
  monomial.  Line (\ref{eqn.psi_1_commute}) follows from
  Cor.~\ref{cor.shift_comm_1} since for each $p$ in the sum, $p \leq
  \ell < 2^m$.
\end{proof}

\begin{lemma}\label{lem.shift_homotopy}
  For $m \geq 0$, if $a_i \geq 2^{m}$ in $x = [a_1, \ldots, a_i,
    \ldots, a_s]$, then
  \[
    x\psi_i^{2^m}Sq^{2^m} + xSq^{2^m}\psi_i^{2^m} = x.
  \]
\end{lemma}
\begin{proof}
  Again, it is sufficient to prove the lemma for $i = 1$ and $s > 1$.
  Write $x = [a_1]\cdot z$ for $z = [a_2, \ldots, a_s]$.  Consider the
  two terms separately:
  \begin{eqnarray}
    \nonumber
    x\psi_1^{2^m}Sq^{2^m} &=& ([a_1] \cdot z)\psi_1^{2^m}Sq^{2^m}\\
    &=& ([a_1]\psi_1^{2^m} \cdot z)Sq^{2^m}
    \label{eqn.shift1}\\
    &=& \sum_{p=0}^{2^m}[a_1]\psi_1^{2^m}Sq^{p} \cdot zSq^{2^m - p}
    \label{eqn.cartan1}\\
    &=& \left(\sum_{p=0}^{2^m - 1}[a_1]\psi_1^{2^m}Sq^{p} \cdot
      zSq^{2^m - p}\right) + [a_1]\psi_1^{2^m}Sq^{2^m} \cdot z.
    \label{eqn.last_first_term}
  \end{eqnarray}
  \begin{eqnarray}
    \nonumber
    xSq^{2^m}\psi_1^{2^m} &=& ([a_1] \cdot z)Sq^{2^m}\psi_1^{2^m}\\
    &=& \sum_{p=0}^{2^m}\left([a_1]Sq^p \cdot zSq^{2^m-p}\right)\psi_1^{2^m}
    \label{eqn.shift2}\\
    &=& \sum_{p=0}^{2^m}[a_1]Sq^p\psi_1^{2^m} \cdot zSq^{2^m-p}\\
    &=& \left(\sum_{p=0}^{2^{m}-1}[a_1]\psi_1^{2^m}Sq^p \cdot
    zSq^{2^m-p}\right) + [a_1]Sq^{2^m}\psi_1^{2^m} \cdot z.
    \label{eqn.last_second_term}
  \end{eqnarray}
  In line~(\ref{eqn.last_second_term}) we just seperate the sum into
  those terms for which $p < 2^m$ (so that $Sq^p$ commutes with
  $\psi_1^{2^m}$), and the top term for which $p = 2^m$.  Add
  lines~(\ref{eqn.last_first_term}) and~(\ref{eqn.last_second_term}),
  and note that the summations cancel completely:
  \begin{eqnarray}
    x\psi_i^{2^m}Sq^{2^m} + xSq^{2^m}\psi_i^{2^m}
    &=& [a_1]\psi_1^{2^m}Sq^{2^m} \cdot z
    + [a_1]Sq^{2^m}\psi_1^{2^m} \cdot z\\
    &=& \left([a_1]\psi_1^{2^m}Sq^{2^m}
    + [a_1]Sq^{2^m}\psi_1^{2^m}\right)\cdot z\\
    &=& [a_1]\cdot z\label{eqn.homotopy_Lemma_9}\\
    &=& x.
  \end{eqnarray}
  Line~(\ref{eqn.homotopy_Lemma_9}) follows from
  Lemma~\ref{lem.shift_homotopy_1} and completes the proof.
\end{proof}

\section{Homotopy Systems in some $\mathcal{A}$-modules}
\label{sec.hom_sys_general}

\subsection{Homotopy Systems in $\widetilde{\Gamma}$}
The results of Lemmas~\ref{lem.shift_commutation}
and~\ref{lem.shift_homotopy} show that there is a homotopy system at
work in $\widetilde{\Gamma}$.  Define for each $1 \leq i \leq s$, and
$k \geq 0$, a graded subspace, $N_s(i,k) \subseteq
\widetilde{\Gamma}_{s,*}$,
\[
  N_s(i,k) = \mathrm{span}\{ [a_1, \ldots, a_i, \ldots, a_s] \in
  \widetilde{\Gamma}_{s_*} \;|\; a_i \geq 2^{k} \}.
\]
\begin{prop}\label{prop.hom_sys_Gamma}
  Fix $s \geq 1$ and $k \geq 0$.  For each $1 \leq i \leq s$, the set
  of shift maps $\{\psi_i^{2^m}\}_{m \leq k}$ forms a $k^{th}$-order
  homotopy system of $\widetilde{\Gamma}_{s,*}$ with null subspace
  $N_s(i,k)$.
\end{prop}
\begin{proof}
  $N_s(i,k)$ is stable under $\psi^{2^m}$, since the effect of the
  shift map is to increase the index in position $i$.  The
  verification of Eqns.~(\ref{eqn.comm_diag}) and~(\ref{eqn.homotopy})
  has been done on generators of $N_s(i,k)$ in
  Lemmas~\ref{lem.shift_commutation} and~\ref{lem.shift_homotopy}.
\end{proof}
As an immediate corollary, we obtain:
\begin{theorem}\label{thm.im-ker_Gamma}
  Let $s \geq 1$ and $k \geq 0$.  If $x \in N_s(i,k)$ for any $1 \leq
  i \leq s$, then $x \in \Delta(k)$ if and only if $x \in I(k)$.
\end{theorem}
\begin{proof}
  This follows from Prop.~\ref{prop.hom_sys_Gamma} and
  Thm.~\ref{thm.imker}.
\end{proof}  

\subsection{Localizations of $H^*(BV_s)$}

Consider $H^*(BV_1, \F_2) \cong \F_2[t]$.  The localization at the
prime $t$, denoted $\Lambda_1$, is an object studied in
Adams~\cite{Adams}, Singer~\cite{Sing3}, and elsewhere.  Note that
$\Lambda_1$ is a left $\mathcal{A}$-algebra, and so its graded vector
space dual, $\nabla_1 = \Lambda_1^*$, is a right
$\mathcal{A}$-algebra.  Writing $[a] \in \nabla_1$ for the element
dual to $t^a \in \Lambda_1$, we find the explicit action of the
Steenrod squares on $\nabla_1$:
\begin{equation}\label{eqn.nabla}
  [a]Sq^i = \left(\!\!\binom{a-i}{i}\!\!\right)[a-i],
\end{equation}
where the notation $\left(\!\binom{n}{i}\!\right)$ stands for the
coefficient of $x^i$ in the formal power series $(1 + x)^n = \sum_{i
  \geq 0} \left(\!\binom{n}{i}\!\right) x^i$, for arbitrary $n \in
\Z$.

Now for each $s \geq 1$, let $\nabla_s$ be the $s$-fold tensor product
of $\nabla_1$ with itself:
\[
  \nabla_s = \left(\nabla_1\right)^{\otimes s} \cong \mathrm{span}\{
        [a_1, a_2, \ldots, a_s] \;|\; a_i \in \Z\},
\]
where the notation $[a_1, a_2, \ldots, a_s]$ represents a monomial of
non-commuting symbols, $[a_1] \otimes [a_2] \otimes \cdots \otimes
[a_s]$.  Formula~(\ref{eqn.nabla}) is extended to elements of
$\nabla_s$ via the Cartan formula so that the graded space $\nabla =
\{ \nabla_{s} \}_{s \geq 0}$ becomes a graded $\mathcal{A}$-algebra.
We found that $\nabla_s$ has the remarkable property that it admits
homotopy systems of all orders each of which has null subspace
consisting of the entire space!

\begin{theorem}\label{thm.im-ker_nabla_s}
  Let $s \geq 1$ and $k \geq 0$ and set $M = \nabla_s$.  Then
  $\Delta_M(k) = I_M(k)$.
\end{theorem}
\begin{proof}
  Consider the shift system, $\{\psi_1^{2^{m}}\}_{m \leq k}$ defined
  in Section~\ref{sec.shift}.  The definition of $\psi_1^{2^m}$ can be
  extended to elements of $\nabla_s$ in the straightforward way,
  namely, $\psi_1^{2^m}$ increases the first index of $[a_1, \ldots,
    a_s] \in \nabla_s$ by $2^m$.  Now Eqn.~(\ref{eqn.binom})
  generalizes neatly:
  \begin{equation}\label{eqn.binom_generalized}
    \textrm{For $a \in \Z$,} \quad \left(\!\!\binom{a}{2^n}\!\!\right)
    = \left\{ \begin{array}{ll} 0, \;\mathrm{if} \; a \equiv 0, 1,
      \ldots, 2^n-1 \; (\mathrm{mod}\; 2^{n+1}),\\ 1, \;\mathrm{if} \;
      a \equiv 2^n, 2^n+1, \ldots, 2^{n+1}-1 \;(\mathrm{mod}\;
      2^{n+1}).\\
    \end{array}\right.
  \end{equation}
  Since there is now no restriction on the value of $a$, we obtain
  unrestricted relations in $\nabla_1$.  For any $a \in \Z$,
  \begin{eqnarray*}
    {[a]}\psi_1^{2^m}Sq^{\ell} &=& [a]Sq^{\ell}\psi_1^{2^m}, \quad
    \textrm{if $2^m > \ell \geq 0$}\\
    {[a]}\psi_1^{2^m}Sq^{2^m} + [a]Sq^{2^m}\psi_1^{2^m} &=& [a]\\
  \end{eqnarray*}
  The relations are easily shown on $\nabla_s$ in general.  For
  any $x \in \nabla_s$,
  \begin{eqnarray*}
    x\psi_1^{2^m}Sq^{\ell} &=& xSq^{\ell}\psi_1^{2^m}, \quad
    \textrm{if $2^m > \ell \geq 0$}\\
    x\psi_1^{2^m}Sq^{2^m} + xSq^{2^m}\psi_1^{2^m} &=& x\\
  \end{eqnarray*}
  Thus there is a $k^{th}$-order homotopy system on $\nabla_s$ with
  null subspace the entire space $\nabla_s$.
\end{proof}

\subsection{Homotopy Systems in $\F_2 \otimes_G \widetilde{\Gamma}_{s,*}$}

For fixed $s \geq 1$, consider a subgroup $G$ of $GL(s, \F_2)$.  The
$\mathcal{A}^+$-annihilated problem in $\F_2 \otimes_G
\widetilde{\Gamma}_{s,*}$ is one of current interest, being dual to
the hit problem for the subspace $H^*\left(\prod^s \R P^{\infty}, \F_2
\right)^G$ of $G$-invariant elements~\cite{JW,Si}.  When $G$ is a
non-trivial subgroup of $GL(s, \F_2)$, the $i^{th}$-place shift maps
$\{\psi_i^{2^m}\}$ defined for $\widetilde{\Gamma}_{s,*}$ may not be
well-defined on $\F_2 \otimes_G \widetilde{\Gamma}_{s,*}$.  However,
certain results can still be found.  First consider the case $G =
\Sigma_s$.  Denote by $\langle a_1, a_2, \ldots, a_s \rangle$ the
image of $[a_1, a_2, \ldots, a_s]$ under the canonical projection
$\widetilde{\Gamma}_{s,*} \to \F_2 \otimes_{\Sigma_s}
\widetilde{\Gamma}_{s,*}$.  Note, the $a_i$'s may freely change
positions in the expression $\langle a_1, a_2, \ldots, a_s \rangle$.
We make the convention that any such term has $a_1 \geq a_2 \geq
\ldots \geq a_s$.
\begin{theorem}\label{thm.im-ker_Gamma_comm}
  Let $s \geq 1$ and $k \geq 0$, and set $M = \F_2 \otimes_{\Sigma_s}
  \widetilde{\Gamma}_{s,*}$.  Suppose $x \in M$ is a sum of terms of
  the form $\langle a_1, a_2, \ldots, a_s\rangle$ such that $a_1 - a_2
  \geq 2^k$.  Then $x \in \Delta_M(k)$ if and only if $x \in I_M(k)$.
\end{theorem}
\begin{proof}
  We first define a homotopy system on $\F_2 \otimes_{\Sigma_s}
  \widetilde{\Gamma}_{s,*}$.  For $m \leq k$, define $\psi^{2^m}$ on
  generators:
  \[
    \langle a_1, a_2, \ldots, a_s\rangle\psi^{2^m} = \langle a_1+2^m,
    a_2, \ldots, a_s\rangle
  \]
  The map is well-defined because of the convention that $a_1 \geq a_2
  \geq \cdots \geq a_s$.  Now since $a_1 \geq 2^k + a_2$,
  Lemmas~\ref{lem.shift_commutation} and~\ref{lem.shift_homotopy}
  should apply to the shift map we have defined, but there is a
  subtlety here.  After applying $Sq^{2^m}$, the values can decrease
  (and hence be permuted so as to maintain the convention of writing
  the values in non-increasing order), but if $a_1 - a_2 \geq 2^k$,
  then all terms of $\langle a_1, a_2, \ldots, a_s\rangle Sq^{2^m}$
  (for $m \leq k$) are of the form $\langle b_1, b_2, \ldots,
  b_s\rangle$, where $b_1 \geq a_2 \geq b_2, b_3, \ldots, b_s$, and so
  the shift operation will act on the correct index.  This dictates
  that the appropriate null subspace is the span of elements of the
  form $\langle a_1, a_2, \ldots, a_s\rangle$ such that $a_1 - a_2
  \geq 2^k$.  Clearly, this space is invariant under the shift
  operations $\psi^{2^m}$ as defined in this proof.
\end{proof}

A similar trick can be used when $G = C_s$, the cyclic group of order
$s$.  Denote by $(a_1, a_2, \ldots, a_s)$ the image of $[a_1, a_2,
  \ldots, a_s]$ under the canonical projection
$\widetilde{\Gamma}_{s,*} \to \F_2 \otimes_{C_s}
\widetilde{\Gamma}_{s,*}$.  In the term $(a_1, a_2, \ldots, a_s)$, the
values may be cyclically permuted, and so generators of $\F_2
\otimes_{C_s} \widetilde{\Gamma}_{s,*}$ can be taken to have the form
$(a_1, a_2, \ldots, a_s)$ where $a_1$ is a maximal value in the list.
There is still some ambiguity as there could be multiple maximal
values in the list $\{a_i\}$, and so we need to make a definitive
choice in such cases.  Although the exact choice we make will be
immaterial in the following arguments, we nevertheless must specify a
consistent convention.  Let $m = \mathrm{max}\{a_i\}$, and choose the
cyclic order of $\{a_i\}$ (given by an element $g \in C_s$) such that
when $(a_{g(1)}, a_{g(2)}, \ldots, a_{g(s)})$ is read left-to-right as
a number in base $m+1$, that number is maximal among all such numbers
for various cyclic orders of the $\{a_i\}$.
\begin{theorem}\label{thm.im-ker_Gamma_cyclic}
  Let $s \geq 1$ and $k \geq 0$, and set $M = \F_2 \otimes_{C_s}
  \widetilde{\Gamma}_{s,*}$.  Suppose $x \in M$ is a sum of terms of
  the form $(a_1, a_2, \ldots, a_s)$ such that for all $i > 1$, we
  have $a_1 - a_i > 2^k$.  Then $x \in \Delta_M(k)$ if and only if $x
  \in I_M(k)$.
\end{theorem}
\begin{proof}
  Use the homotopy system $\{\psi_1^{2^m}\}_{m \leq k}$, defined
  originally for $\widetilde{\Gamma}_{s,*}$ (applied to generators of
  $M$ written in the order specified above).  The requirement that
  $a_1$ is at least $2^k$ greater than any other $a_i$ ensures that
  the squaring operations do not ``mix up'' the order of factors in
  each term.
\end{proof}

\section{The Quotient $\Delta_M(k)/I_M(k)$}\label{sec.quotient}

For any right $\mathcal{A}$-module $M$, $I_M(k)$ is a vector subspace
of $\Delta_M(k)$.  However, when $M$ is an $\mathcal{A}$-algebra, much
more can be said.  The following was observed by
Singer~\cite{Singer_private}.
\begin{prop}\label{I_M(k)_ideal}
  If $M$ is a right $\mathcal{A}$-algebra, then for each $k \geq 0$,
  $I_M(k)$ is a two-sided ideal of the algebra $\Delta_M(k)$.
\end{prop}
\begin{proof}
  The fact that $\Delta_M(k)$ is an algebra follows easily from the
  Cartan formula.  Now let $x \in I_M(k)$ and $y \in \Delta_M(k)$.
  For each $i \leq k$, write $x = z_iSq^{2^{i+1}-1}$.  Note, for $i =
  0, 1, 2, \ldots, k$, $Sq^{2^{i+1}-1} \in \mathcal{S}^+_k$.
  Therefore, since $y$ is $\mathcal{S}^+_k$-annihilated, we have for
  each $i \leq k$:
  \[
    yx = y\cdot z_iSq^{2^{i+1}-1} = (yz_i)Sq^{2^{i+1}-1},
  \]
  which proves $yx \in I_M(k)$.  Similarly, $xy \in I_M(k)$.
\end{proof}

\begin{definition}
  For each $k \geq 0$, the space of {\it un-hit
    $\mathcal{S}^+_k$-annihilateds} is the quotient of spaces,
  \[
    U_M(k) = \Delta_M(k)/I_M(k).
  \]
  If $M$ is graded, then $U_M(k)$ inherits this grading.  If $M$ is an
  $\mathcal{A}$-algebra, then $U_M(k)$ is an algebra with product
  induced from the product in $M$.
\end{definition}

\begin{rmk}
  When $M$ is the homology of an associative $H$-space, then the
  Pontryagin product makes $M$ into a right $\mathcal{A}$-algebra, and
  hence each $U_M(k)$ is an algebra.  On the other hand, if $M =
  \{M_{s,*}\}_{s \geq 0}$ is bigraded with $M_{s,*} = H_*(X_s, \F_2)$,
  where $\{X_s\}_{s \geq 0}$ is a family of spaces equipped with an
  associative pairing $X_s \times X_t \to X_{s+t}$, then $M$ admits
  the structure of bigraded right $\mathcal{A}$-algebra, and so each
  $U_M(k)$ is a bigraded algebra.  It is interesting to see that the
  bigraded $\widetilde{\Gamma} = \{\widetilde{\Gamma}_{s,*}\}_{s \geq
    0}$ possesses both types of multiplication, and so
  $U_{\tilde{\Gamma}}(k)$ is a right $\mathcal{A}$-algebra in two
  very different ways!
\end{rmk}

\section{Fine Structure of $\Delta(1)$}\label{sec.structure}

For the remainder of this paper, we examine the structure of
$\Delta(1) = \mathrm{ker}\,Sq^1 \cap \mathrm{ker}\,Sq^2 \subseteq
\widetilde{\Gamma}$.  We will eventually produce an element $z \in
\Delta(1)_{5,9}$ that is not in the image of $Sq^3$, hence not in
$I(1)_{5,9}$.  In this section, fix $s \geq 2$, $d \geq s$, and write
$x \in \widetilde{\Gamma}_{s,d}$ in the form $x = \sum_{i \geq 1}
      [i]\cdot x_i$ for elements $x_i \in
      \widetilde{\Gamma}_{s-1,d-i}$.  The following two propositions
      characterize what it means for $x$ to be in
      $\mathrm{ker}\,Sq^1$, {\it resp.} $\mathrm{ker}\,Sq^2$.
\begin{prop}\label{prop.Sq^1}
  If $x = \sum_{i \geq 1} [i]\cdot x_i \in \mathrm{ker}\,Sq^1$,
  then for each $n \geq 1$,
  \begin{eqnarray}
    x_{2n} &=& x_{2n-1}Sq^1, \quad \textrm{and}\label{eqn.Sq^1-1}
    \label{eqn.2n-1}\\
    x_{2n}Sq^1 &=& 0. \label{eqn.Sq^1-2}
    \label{eqn.2n}
  \end{eqnarray}
\end{prop}
\begin{rmk}
  Of course, Eqn.~(\ref{eqn.Sq^1-2}) follows as a consequence of
  Eqn.~(\ref{eqn.Sq^1-1}).
\end{rmk}
\begin{proof}
  Apply $Sq^1$ to both sides of $x = \sum_{i \geq 1} [i] \cdot x_i$.
  \begin{eqnarray*}
    0 &=& \left(\sum_{i \geq 1} [i] \cdot x_i\right)Sq^1 \\
    &=& \sum_{n \geq 1} [2n-1] \cdot x_{2n-1}Sq^1
    + \sum_{n \geq 1} \left( [2n-1] \cdot x_{2n} 
    + [2n] \cdot x_{2n}Sq^1\right)\\
    &=& \sum_{n \geq 1} [2n-1] \cdot \left( x_{2n-1}Sq^1 + x_{2n} \right)
    + \sum_{n \geq 1} [2n] \cdot x_{2n}Sq^1.
  \end{eqnarray*}
  Since $\widetilde{\Gamma}_{s,*}$ is a free algebra on the generators
  $\{[i]\}_{i \geq 1}$, the above computation proves~(\ref{eqn.2n-1})
  and~(\ref{eqn.2n}).
\end{proof}

\begin{prop}\label{prop.Sq^2}
  If $x = \sum_{i \geq 1} [i]\cdot x_i \in \mathrm{ker}\,Sq^2$, then
  for each $m \geq 1$,
  \begin{eqnarray}
    x_{4m-2}Sq^1 &=& x_{4m-3}Sq^2, \label{eqn.Sq^2-1}\\
    x_{4m} &=& x_{4m-2}Sq^2, \label{eqn.Sq^2-2}\\
    x_{4m+1} &=& x_{4m-1}Sq^2 + x_{4m}Sq^1, \quad \textrm{and}
    \label{eqn.Sq^2-3}\\
    x_{4m}Sq^2 &=& 0. \label{eqn.Sq^2-4}
  \end{eqnarray}
\end{prop}
\begin{proof}
  Apply $Sq^2$ to both sides of $x = \sum_{i \geq 1} [i] \cdot x_i$.
  \[
    0 = \left(\sum_{i \geq 1} [i] \cdot x_i\right)Sq^2
  \]
  We then expand using the Cartan formula and perform many tedious
  (though elementary) manipulations on the summation.  The details are
  left to the diligent reader.  Care must be taken with the low order
  terms, but eventually we find:
  \begin{eqnarray}
    0 &=& [1] \cdot \left(x_1Sq^2 + x_2Sq^1\right) +
    [2] \cdot \left(x_2Sq^2 + x_4\right)
    + [3] \cdot \left(x_3Sq^2 + x_4Sq^1 + x_5\right)
    \nonumber\\
    && {}+ \sum_{m \geq 1} [4m] \cdot x_{4m}Sq^2 
    \nonumber\\
    && {} + \sum_{m \geq 1} [4m+1] \cdot \left(x_{4m+1}Sq^2 + 
    x_{4m+2}Sq^1 \right)
    \nonumber\\
    && {} + \sum_{m \geq 1} [4m+2] \cdot \left(x_{4m+2}Sq^2 +
    x_{4m+4}\right)
    \nonumber\\
    && {} + \sum_{m \geq 1} [4m+3] \cdot \left(x_{4m+3}Sq^2 +
    x_{4m+4}Sq^1 + x_{4m+5} \right)
    \nonumber
  \end{eqnarray}
  Including each of the low order terms into the appropriate
  summation, the individual summations then give each of the
  relations,~(\ref{eqn.Sq^2-1}),~(\ref{eqn.Sq^2-2}),~(\ref{eqn.Sq^2-3}),
  and~(\ref{eqn.Sq^2-4}).
\end{proof}

\begin{prop}\label{prop.structure}
  Suppose $x \in \widetilde{\Gamma}_{s,*}$ for some $s \geq 2$.  Then
  $x \in \Delta(1)$ if and only if $x = \sum_{i \geq 1} [i]\cdot x_i$,
  where the $x_i \in \widetilde{\Gamma}_{s-1,*}$ satisfy:
  \begin{eqnarray}
    x_1 &\in& \mathrm{ker}\,Sq^2, \label{eqn.Delta(1)-1}\\
    x_2 &=& x_1Sq^1, \label{eqn.Delta(1)-2}\\
    x_3 &\in& \left(Sq^1\right)^{-1}(x_1Sq^3), \label{eqn.Delta(1)-3}\\
    \textrm{and for each $m \geq 1$,} \qquad\qquad
    x_{4m} &=& x_{4m-1}Sq^1,\label{eqn.Delta(1)-4}\\
    x_{4m+1} &=& x_{4m-1}Sq^2, \label{eqn.Delta(1)-5}\\
    x_{4m+2} &=& x_{4m-1}Sq^2Sq^1, \label{eqn.Delta(1)-6}\\
    x_{4m+3} &\in& \left( Sq^1 \right)^{-1}\left[x_{4m-1}Sq^2Sq^3\right].
    \label{eqn.Delta(1)-7}
  \end{eqnarray}
\end{prop}
\begin{proof}
  Fix $s \geq 2$, and let $x \in \Delta(1)_{s,*}$.  Then $xSq^1 =
  xSq^2 = 0$.  Prop.~\ref{prop.Sq^1} gives $x_2 = x_1Sq^1$ and for
  each $m \geq 1$, $x_{4m} = x_{4m-1}Sq^1$,
  proving~(\ref{eqn.Delta(1)-2}) and~(\ref{eqn.Delta(1)-4}).  Next,
  with $m=1$,~(\ref{eqn.Sq^2-1}) gives $x_1Sq^2 = x_2Sq^1$.
  Then~(\ref{eqn.2n-1}) can be used: $x_2Sq^1 = (x_1Sq^1)Sq^1 = 0$,
  proving~(\ref{eqn.Delta(1)-1}).  In order to
  prove~(\ref{eqn.Delta(1)-3}), all that we must show is that $x_3Sq^1
  = x_1Sq^3$.  This is done by using~(\ref{eqn.Sq^2-2}) with $m=1$:
  $x_4 = x_2Sq^2$, and then applying~(\ref{eqn.Sq^1-1}) to get:
  $x_3Sq^1 = x_1Sq^1Sq^2 = x_1Sq^3$.  Next,~(\ref{eqn.Sq^2-3})
  together with~(\ref{eqn.Sq^1-2}) yields~(\ref{eqn.Delta(1)-5}).
  Apply $Sq^1$ to both sides of~(\ref{eqn.Delta(1)-5}) to
  get~(\ref{eqn.Delta(1)-6}).  Finally, to
  prove~(\ref{eqn.Delta(1)-7}), observe (for any $m \geq 1$):
  \begin{eqnarray*}
    x_{4m+4} &=& x_{4m+2}Sq^2, \quad 
    \textrm{by~(\ref{eqn.Sq^2-2})}\\
    x_{4m+3}Sq^1 &=& x_{4m+1}Sq^1Sq^2, \quad 
    \textrm{by~(\ref{eqn.Sq^1-1})}\\
    x_{4m+3}Sq^1 &=& x_{4m-1}Sq^2Sq^1Sq^2, \quad 
    \textrm{by~(\ref{eqn.Delta(1)-5})}\\
   &=& x_{4m-1}Sq^2Sq^3.
  \end{eqnarray*}
  This concludes the forward direction of the proof.

  For the reverse direction, we are given $x \in
  \widetilde{\Gamma}_{s,*}$ with $s \geq 2$, and $x = \sum_{i \geq 1}
  [i]\cdot x_i$, with the relations on $x_i$ as stated in
  Prop.~\ref{prop.structure}.  We must show that $xSq^1 = 0$ and
  $xSq^2 = 0$.  This routine verification is omitted.
\end{proof}

\begin{cor}\label{cor.Delta(1)_determined}
  The elements $x = \sum_{i \geq 1} [i]\cdot x_i \in \Delta(1)$ are
  determined by a choice of $x_1 \in \mathrm{ker}\,Sq^2$, a choice of
  $x_3 \in \left(Sq^1\right)^{-1}\left[x_1Sq^3\right]$, and for each
  $m \geq 1$, a choice of $x_{4m+3} \in
  \left(Sq^1\right)^{-1}\left[x_{4m-1}Sq^2Sq^3\right]$.
\end{cor}

\begin{rmk}
  By Prop.~\ref{prop.structure}, $x \in \Delta(1)$ if and only if $x_1
  \in \mathrm{ker}\,Sq^2$ and certain other relations hold on other
  $x_i$.  Corollary~\ref{cor.Delta(1)_determined} states that any
  choice of $x_1 \in \mathrm{ker}\, Sq^2$ can be built up into an
  element $x \in \Delta(1)$, or in other words, the choice of $x_1 \in
  \mathrm{ker}\, Sq^2$ is not restricted by any other relation.
\end{rmk}

\section{What is in $U_M(k)$?}\label{sec.questions}

The machinery of homotopy systems provides a way to categorize large
subspaces of $\mathcal{S}^+_k$-annihilated elements that are in the
image of spike squares.  In a sense, the elements of $I_M(k)$ are easy
to handle, and so our attention turns to those
$\mathcal{S}^+_k$-annihilateds that are not in $I_M(k)$.  In other
words,
\begin{question}
  For a given right $\mathcal{A}$-module $M$, what is in $U_M(k)$?
\end{question}

\begin{example}
  In the case $M = \nabla_s$, Thm.~\ref{thm.im-ker_nabla_s} shows
  $\Delta_M(k) = I_M(k)$ for all $k$.  In other words, $U_{\nabla}(k)
  = 0$ for all $k \geq 0$.  However, this seems to be an extreme case.
\end{example}

\begin{example}
  $U_{\tilde{\Gamma}}(0) = \F_2$, concentrated in bidegree $(0,0)$, as
  a consequence of Thm.~\ref{thm.im-ker_Gamma}.
\end{example}

In any unstable right $\mathcal{A}$-module $M$, there are a number of
elements of $\Delta_M(k)$ that are un-hit ({\it i.e.} fail to be in
$I_M(k)$) only because their degree is so low that the instability
condition prevents their being in the image of higher order squares.
Let us ignore those types of un-hit elements as {\it degenerate}; such
elements should be considered in $\Delta_M(j)$ for some $j < k$.  It
is much more interesting to know if there are un-hit elements in high
degrees.

\begin{example}
  It is an easy exercise to show that for $k \geq 1$, there are only
  degenerate elements in $U_{\tilde{\Gamma}}(k)_{1,*}$.  In
  particular, if $d \geq 2^{k+1}$ then $U_{\tilde{\Gamma}}(k)_{1,d} =
  0$.
\end{example}

\begin{prop}\label{prop.Sq^3_Sq^2}
  Let $x = \sum_{i \geq 1} [i] \cdot x_i \in \Delta(1)$.  Then $x \in
  I(1)$ if and only if $x_1 = wSq^2$ for some $w \in
  \mathrm{ker}\,Sq^3$.
\end{prop}
\begin{proof}
  Let $x \in \Delta(1)_{s,d}$, for some $s \geq 2$.  It suffices to
  prove: $x \in \mathrm{im}\,Sq^3$ if and only if $x_1 = wSq^2$ for
  some $w \in \mathrm{ker}\, Sq^3$.  Write $x = [1] \cdot x_1 + x'$.
  I claim $x' \in \mathrm{im}\,Sq^3$; indeed, there is a preimage:
  \[
    y' = \sum_{j \geq 2} [2j] \cdot x_{2j-3}.
  \]
  It is straightforward to verify that $y'Sq^3 = x'$, using
  Prop.~\ref{prop.structure}.  Thus, it suffices to prove: $[1] \cdot
  x_1 \in \mathrm{im}\, Sq^3$ if and only if $x_1 = wSq^2$ for some $w
  \in \mathrm{ker}\, Sq^3$.

  A $Sq^3$-preimage of $[1] \cdot x_1$ could only have the form $[1]
  \cdot y_1 + [2] \cdot y_2$ for some $y_1, y_2 \in
  \widetilde{\Gamma}_{s-1,*}$.  Consider the effect of applying $Sq^3$
  on this element:
  \begin{eqnarray*}
    \left( [1] \cdot y_1 + [2] \cdot y_2 \right)Sq^3 &=&
    [1] \cdot (y_1Sq^3 + y_2Sq^2) + [2] \cdot y_2Sq^3\\
    &=& [1] \cdot (y_1Sq^1 + y_2)Sq^2 + [2] \cdot y_2Sq^3.
  \end{eqnarray*}
  Thus we find that for $w = y_1Sq^1 + y_2$, $[1] \cdot x_1 \in
  \mathrm{im} \, Sq^3$ implies $wSq^2 = x_1$, and $wSq^3 = y_2Sq^3 =
  0$.  Conversely, if any $w$ exists such that $wSq^2 = x_1$ and
  $wSq^3 = 0$, then $([2]\cdot w)Sq^3 = [1] \cdot x_1$.
\end{proof}

\begin{cor}
  There are only degenerate elements in $U_{\tilde{\Gamma}}(1)_{2,*}$.
\end{cor}
\begin{proof}
  By Cor.~\ref{cor.Delta(1)_determined} all elements of
  $\Delta(1)_{2,*}$ are determined by a choice of $x_1 \in
  \mathrm{ker}\,Sq^2$, together with choices for $x_{4m-1}$, $m \geq
  1$.  By Prop.~\ref{prop.Sq^3_Sq^2}, it suffices to examine only the
  first term $[1] \cdot x_1$ of an element $x \in \Delta(1)$.  Now
  $x_1 \in \widetilde{\Gamma}_{1,*}$, and it is easy to see that
  $\mathrm{ker}\, Sq^2 = \mathrm{im}\, Sq^2$ in
  $\widetilde{\Gamma}_{1,d}$ for $d \geq 2$.  Furthermore, all
  preimages of $Sq^2$ in $\widetilde{\Gamma}_{1,*}$ (namely, elements
  of the form $[4j]$ and $[4j+1]$ for $j \geq 1$), are killed by
  $Sq^3$.
\end{proof}

Note that Theorem~\ref{thm.imker} can only be used to show particular
elements $x \in \Delta_M(k)$ are trivial in $U_M(k)$, not to detect
non-trivial elements, and so it may be very difficult to understand
the true nature of $U_M(k)$.  Initially we found no non-trivial
non-degenerate elements of $U_{\tilde{\Gamma}}(k)$.  However, using
Prop.~\ref{prop.structure}, Cor.~\ref{cor.Delta(1)_determined}, and
Prop.~\ref{prop.Sq^3_Sq^2}, together with machine computations, a
non-trivial element $\overline{z} \in U_{\tilde{\Gamma}}(1)_{5,9}$ can
be produced.

\begin{cor}\label{cor.unhit_z}
  Define:
  \begin{eqnarray*}
    z &=& [ 1, 1, 1, 2, 4 ] + [ 1, 1, 2, 1, 4 ] + [ 1, 1, 2, 4, 1 ] +
    [1, 2, 1, 4, 1 ]+ [ 1, 2, 2, 2, 2 ]\\ && {} + [ 1, 4, 1, 1, 2 ] +
    [ 1, 4, 2, 1, 1 ]+ [ 2, 1, 1, 2, 3 ] + [ 2, 1, 2, 1, 3 ] + [ 2, 1,
      2, 2, 2 ]\\ && {} + [ 2, 1, 2, 3, 1 ]+ [ 2, 2, 1, 2, 2 ] + [ 2,
      2, 1, 3, 1 ] + [ 2, 2, 2, 1, 2 ]+ [ 2, 2, 2, 2, 1 ]\\ && {} + [
      2, 3, 1, 1, 2 ]+ [ 2, 3, 2, 1, 1 ]+ [ 3, 1, 2, 1, 2 ] + [ 3, 1,
      2, 2, 1 ]+ [ 3, 2, 2, 1, 1 ]\\ && {} + [ 4, 1, 1, 1, 2 ] + [ 4,
      1, 1, 2, 1 ]+ [ 4, 1, 2, 1, 1 ] + [ 4, 2, 1, 1, 1 ] + [ 5, 1, 1,
      1, 1 ].
  \end{eqnarray*}
  The element $z \in \widetilde{\Gamma}_{5,9}$ represents a
  non-trivial element of $U_{\tilde{\Gamma}}(1)$.
\end{cor}

\begin{proof}
  Let $w = [1, 1, 2, 4 ] + [ 1, 2, 1, 4 ] + [ 1, 2, 4, 1 ] + [2, 1, 4,
    1 ]+ [ 2, 2, 2, 2 ] + [ 4, 1, 1, 2 ] + [ 4, 2, 1, 1 ]$.  It can be
  verified by hand that $wSq^2 = 0$ and by computer that $w \notin
  \mathrm{im}\,Sq^2$.  By Prop.~\ref{prop.structure} and
  Cor.~\ref{cor.Delta(1)_determined}, $w \in \mathrm{ker}\,Sq^2$ can
  be used to produce an element $z = \sum_{i \geq 1} [i] \cdot z_i \in
  \Delta(1)$ with $z_1 = w$.  Such element $z$ cannot be in the image
  of $Sq^3$ due to Prop.~\ref{prop.Sq^3_Sq^2} (a fact that can be
  independently checked by computer as well).
\end{proof}

\begin{rmk}
  The proof of Cor.~\ref{cor.unhit_z} hinges on the existence of an
  element $w \in \mathrm{ker}\,Sq^2 \setminus \mathrm{im}\, Sq^2$.
  Even though $Sq^2$ does not act as a differential on
  $\widetilde{\Gamma}_{s,*}$ for $s \geq 2$, it is interesting to see
  how closely related $\mathrm{ker} \,Sq^2$ and $\mathrm{im} \,Sq^2$
  really are.  We do not expect $\mathrm{im} \,Sq^2 \subseteq
  \mathrm{im}\, Sq^2$, however it is hard to find counterexamples to
  $\mathrm{ker}\, Sq^2 \subseteq \mathrm{im} \,Sq^2$.  The author
  hopes to explore this relationship in future work.
\end{rmk}


\bibliographystyle{plain}

\end{document}